\documentclass[letterpaper, 10pt, conference]{ieeeconf}
\IEEEoverridecommandlockouts
\usepackage[
    backend=biber,
    style=ieee,
    url=false, 
    doi=false,
    isbn=false,
]{biblatex}
\usepackage{amsmath,amssymb,amsfonts}

\usepackage{amsthm}
\newcommand{\R}{\mathbb{R}}
\DeclareMathOperator{\Proj}{Proj}
\DeclareMathOperator{\prox}{prox}
\newtheorem{theorem}{Theorem}
\newtheorem{lemma}{Lemma}
\newtheorem{definition}{Definition}
\newtheorem{remark}{Remark}
\usepackage{algorithmic}
\usepackage{booktabs}
\usepackage{multirow}
\usepackage{graphicx}
\usepackage{textcomp}
\usepackage[table]{xcolor}
\usepackage[hidelinks]{hyperref}

\begin{document}

\title{Efficient ADMM for Soft-Constrained Quadratic Programs\\via Generalized Slack Projections\\
\thanks{This project was supported by the Toyota Research Institute and the National Science Foundation (Awards
2411369, 2535096). Opinions, findings, conclusions, or recommendations expressed in this material are those of the authors and do not necessarily reflect those of the funding organizations.}%
\thanks{Roland Schwan and Brian Plancher are with Dartmouth College, USA. \texttt{\{roland.schwan, plancher\}@dartmouth.edu}}%
\thanks{Thomas Lew is with Toyota Research Institute, Los Altos, USA. \texttt{thomas.lew@tri.global}}%
\thanks{LLM tools, including Codex and Claude, have been used to improve the grammar and style of several paragraphs in this manuscript. All final text and references were edited and
reviewed by humans.}
}

\author{Roland Schwan, Thomas Lew, Brian Plancher}

\maketitle

\begin{abstract}
Quadratic programs (QPs) are a core building block in many applications. In practice, constraints are often softened with slack variables to guarantee feasibility. The standard approach tackles this by introducing additional decision variables, growing the resulting optimization problem. In this paper, we propose an Alternating Direction Method of Multipliers (ADMM) scheme for a broad class of convex slack penalty functions, including all $\ell_k$-norms, without increasing the problem size. The only change to the ADMM scheme is a modified slack projection function, while the remaining algorithm stays the same. We demonstrate our approach on numerical examples, showing significant speedups, and show its flexibility on an energy optimization problem.
\end{abstract}

\section{Introduction}

Numerical optimization is a crucial component in many modern engineering disciplines, with applications spanning robotics \cite{wensing2024_legged_robots}, automotive systems \cite{waibel2026_laopt}, energy systems \cite{vazquez2014_mpc_power}, and finance \cite{rubinstein2002_portfolio}. QPs in particular form a fundamental building block, arising either directly from the problem formulation or as a submodule of more general algorithms such as sequential quadratic programming \cite{nocedal2006_bible}. We consider QPs of the form
\begin{align} \tag{\textcolor{blue}{\textbf{QP}}} \label{eq:QP}
\min_{x}\ 
\tfrac{1}{2}x^\top Q x + q^\top x 
\ \ \text{s.t.}\ \, 
\ell \leq A x\leq u,
\end{align}
where $x\in\R^n$ are decision variables, $Q\in\R^{n\times n}$ is positive semidefinite, $q\in\R^n$, $A\in\R^{p\times n}$, and $\ell,u\in(\R\cup\{\pm\infty\})^p$ with $\ell\leq u$. Popular algorithms to solve such problems include interior-point methods \cite{schwan2023_piqp, goulart2026_clarabel} as well as proximal methods, e.g., ADMM \cite{boyd2011_admm, stellato2020_osqp, odonoghue2016_scs}. We focus on the latter, which is particularly popular due to its simplicity, warm-startability, and cheap iterations.

In many applications, the resulting \ref{eq:QP} might become infeasible due to model inaccuracies or sensor noise. Hence, practitioners often relax the hard constraints to guarantee that \ref{eq:QP} always admits a feasible solution. By introducing slack variables $\xi\in\R^p$, we can form the soft-constrained QP
\begin{equation} \tag{\textcolor{blue}{\textbf{QPslack}}} \label{eq:QP_slack}
\begin{aligned}
\min_{x,\xi}\quad
&\tfrac{1}{2}x^\top Qx+q^\top x+\Phi(\xi)\\
\text{s.t.}\quad
&\ell\leq Ax+\xi\leq u,
\end{aligned}
\end{equation}
where $\Phi:\R^p\to\R\cup\{+\infty\}$ is a proper, closed, convex penalty function with $\Phi(0)=0$. The choice of $\Phi$ encodes how constraint violations are traded off: quadratic penalties spread small violations over many constraints, $\ell_1$ penalties are exact if the penalty weight exceeds the dual norm of the optimal multipliers \cite{kerrigan2000_soft}, and $\ell_\infty$ penalties only charge the peak violation.

\begin{figure}[t]
\centering
\includegraphics[width=\columnwidth]{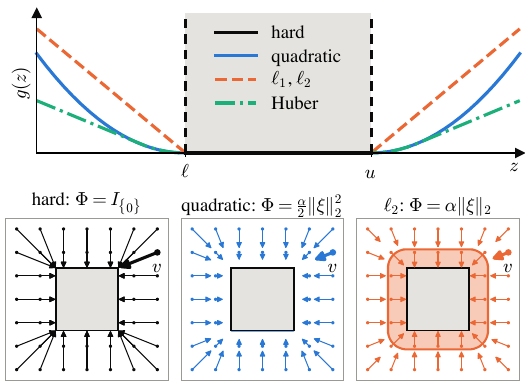}
\caption{Top: The relaxed indicator $g(z)=\inf_{\xi\in\R^p}\bigl\{\Phi(\xi)+I_{[\ell,u]}(z+\xi)\bigr\}$ replaces the hard walls of the box indicator $I_{[\ell,u]}$ with relaxed penalty functions. Bottom: The resulting generalized slack projection in two dimensions. Arrows map points $v$ to $\Proj^{\Phi}_{[\ell,u]}(v)$ \eqref{eq:slack_projection}; in the $\ell_2$ panel (right), the shaded region is mapped exactly onto the box.}
\label{fig:slack_projection}
\vspace{-5mm}
\end{figure}

However, penalties can be costly to handle with standard solvers. The standard approach to solve \ref{eq:QP_slack} is to introduce $\xi_i$ as decision variables and reformulate the problem as a standard \ref{eq:QP}, where nonsmooth penalties require additional variables and constraints \cite{boyd2004_convex}. For example, using the $\ell_1$ penalty $\Phi(\xi)=\alpha\|\xi\|_1$ with $\alpha>0$ gives the lifted problem
\begin{equation}\label{eq:l1_lifting}
\begin{aligned}
\min_{x,s}\quad
&\tfrac{1}{2}x^\top Qx+q^\top x+\alpha\mathbf{1}^\top s\\
\text{s.t.}\quad
&\ell-s\leq Ax\leq u+s,\quad s\geq0,
\end{aligned}
\end{equation}
where $s\in\R^p$ are additional decision variables. Since $s$ enters both bounds with opposite signs, each two-sided constraint is split into two one-sided constraints, and $s\geq0$ adds further constraints (see Remark~\ref{rem:lifting}). As summarized in Table~\ref{tab:lifting}, for $m$ soft rows, lifting increases the order of the KKT system solved in every ADMM iteration from $n+p$ up to $n+p+4m$ and, consequently, the solve time. For non-separable penalties such as the $\ell_\infty$ penalty, lifting can also destroy the sparsity structure needed for efficient solves of multistage problems such as model predictive control (MPC) \cite{schwan2025_multi_piqp}, since the shared slack variable couples all stages together. Additionally, penalty functions like the non-squared Euclidean norm must be reformulated as a second-order cone program, changing the problem class.

Structure-exploiting interior-point solvers such as \texttt{HPIPM} \cite{frison2020_hpipm} avoid the growth of the KKT system by eliminating the slack variables in the linear system. For ADMM, \cite{lew2026_slackqp} showed that squared two-norm penalties can be handled without lifting by modifying the projection step. Moreover, the proximal operator of the distance function to a convex set, which corresponds to the non-squared $\ell_2$ penalty, is known in closed form \cite{parikh2014_prox, beck2017_first_order}. However, to the best of our knowledge, no ADMM scheme handles general penalty functions without lifting.

\noindent \textbf{Contributions}: We extend the ADMM approach in \cite{lew2026_slackqp} from squared two-norm penalty functions to general \emph{unconditional} convex penalties, i.e., penalties invariant under coordinatewise sign changes. This class includes all $\ell_k$-norms and separable penalties with even convex components, in particular all penalties in Table~\ref{tab:lifting}. As illustrated in Fig.~\ref{fig:slack_projection}, the only change to the ADMM scheme is a generalized slack projection. This allows us to solve \ref{eq:QP_slack} \textit{without} introducing additional decision variables, resulting in faster solve times. We implement our approach as an extension of the popular open-source solver \texttt{OSQP} \cite{stellato2020_osqp} and show speedups of up to $3.4\times$ over lifting on random sparse QPs. In addition, we give an application to building energy management, motivating the use of different types of penalty functions. The code to reproduce the results is available at\\\texttt{\url{https://github.com/A2R-Lab/OSQPSlack}}.

\section{Solving \texorpdfstring{\ref{eq:QP}}{QP} with ADMM} \label{sec:QP_ADMM}

Our method builds on the ADMM splitting of \cite{stellato2020_osqp}, which we review below. For clarity, we omit details such as over-relaxation and adaptive $\rho$ updates, as including them does not impact our results and follows directly from prior work~\cite{stellato2020_osqp}.

Introducing the auxiliary decision variables $\tilde x \in \R^n$ and $z, \tilde z \in \R^p$, the \ref{eq:QP} is equivalent to
\begin{equation}\label{eq:QP_OSQP_split}
\begin{aligned}
\min_{\tilde x,\tilde z,x,z}\quad
&\tfrac{1}{2}\tilde x^\top Q\tilde x
+q^\top\tilde x
+I_{A\tilde x=\tilde z}(\tilde x,\tilde z)
+I_{[\ell,u]}(z)\\
\text{s.t.}\quad
&(\tilde x,\tilde z)=(x,z),
\end{aligned}
\end{equation}
where $I_\mathcal{C}(\cdot)$ is the indicator function
\begin{equation}\label{eq:indicator}
I_\mathcal{C}(z) =
\begin{cases}
0 & \text{if } z \in \mathcal{C}, \\
+\infty & \text{otherwise,}
\end{cases}
\end{equation}
which separates the cost and linear equalities from the inequality constraints.
For step-size parameters $\sigma,\rho>0$, the augmented Lagrangian of \eqref{eq:QP_OSQP_split} is
\begin{equation}\label{eq:QP_augmented_lagrangian}
\begin{aligned}
\mathcal L_{\sigma,\rho}
(\tilde x,\tilde z,x,z,\lambda,\mu)
={}&\tfrac{1}{2}\tilde x^\top Q\tilde x
+q^\top\tilde x\\
&+I_{A\tilde x=\tilde z}(\tilde x,\tilde z)
+I_{[\ell,u]}(z)\\
&+\lambda^\top(\tilde x-x)
+\tfrac{\sigma}{2}\|\tilde x-x\|_2^2\\
&+\mu^\top(\tilde z-z)
+\tfrac{\rho}{2}\|\tilde z-z\|_2^2,
\end{aligned}
\end{equation}
where $\lambda \in \R^n$ and $\mu \in \R^p$ are the associated Lagrange multipliers.
ADMM then solves \ref{eq:QP} by applying the following iterative algorithm to \eqref{eq:QP_augmented_lagrangian}:
\begin{subequations}\label{eq:QP_ADMM}
\begin{align}
(\tilde x^{k+1},\tilde z^{k+1})
&=\mathop{\arg\min}_{\tilde x,\tilde z}
\mathcal L_{\sigma,\rho}
(\tilde x,\tilde z,x^k,z^k,\lambda^k,\mu^k), \label{eq:QP_ADMM_step1}\\
(x^{k+1},z^{k+1})
&=\mathop{\arg\min}_{x,z}
\mathcal L_{\sigma,\rho}
(\tilde x^{k+1},\tilde z^{k+1},x,z,\lambda^k,\mu^k), \label{eq:QP_ADMM_step2}\\
\lambda^{k+1}
&=\lambda^k+\sigma(\tilde x^{k+1}-x^{k+1}),\\
\mu^{k+1}
&=\mu^k+\rho(\tilde z^{k+1}-z^{k+1}).
\end{align}
\end{subequations}

\noindent \textbf{$(\tilde x,\tilde z)$-update} \eqref{eq:QP_ADMM_step1}:
Let $\nu^{k+1}$ be the Lagrange multiplier for the constraint
$A\tilde x=\tilde z$ in the first step. Its optimality conditions are
\begin{subequations}\label{eq:QP_tilde_optimality}
\begin{align}
0={}&Q\tilde x^{k+1}+q+\lambda^k
+\sigma(\tilde x^{k+1}-x^k)+A^\top\nu^{k+1},\\
0={}&\mu^k+\rho(\tilde z^{k+1}-z^k)-\nu^{k+1},\\
0={}&A\tilde x^{k+1}-\tilde z^{k+1}.
\end{align}
\end{subequations}

\noindent \textbf{$(x,z)$-update} \eqref{eq:QP_ADMM_step2}:
The second step separates as
\begin{subequations}\label{eq:QP_copy_solution}
\begin{align}
x^{k+1}
&=\tilde x^{k+1}+\sigma^{-1}\lambda^k,
\label{eq:copy_x}\\
z^{k+1}
&=\mathop{\arg\min}_z\left\{
I_{[\ell,u]}(z)
+\tfrac{\rho}{2}
\|z-v^{k+1}\|_2^2
\right\} \notag\\
&=\Proj_{[\ell,u]}
(v^{k+1}), \label{eq:z_update}
\end{align}
\end{subequations}
where $v^{k+1}:=\tilde z^{k+1}+\rho^{-1}\mu^k$ and $\Proj_{[\ell,u]}(\cdot)$ is the elementwise projection operator
\begin{equation}
\Proj_{[\ell,u]}(z) =
\begin{cases}
z & \text{if } \ell \leq z \leq u, \\
\ell & \text{if } z < \ell, \\
u & \text{if } z > u.
\end{cases}
\end{equation}

\begin{table}[t]
\caption{Problem size of \ref{eq:QP_slack} with $m$ two-sided soft rows. Lifting to a standard \ref{eq:QP} grows the problem, whereas our approach retains the size of the hard QP for all penalties.}
\label{tab:lifting}
\centering
\renewcommand{\arraystretch}{1.2}
\begin{tabular}{lcccc}
\toprule
& \multicolumn{2}{c}{Variables} & \multicolumn{2}{c}{KKT order} \\
\cmidrule(lr){2-3}\cmidrule(lr){4-5}
Penalty & Ours & Lifted & Ours & Lifted \\
\midrule
Quadratic & $n$ & $n+m$ & $n+p$ & $n+p+m$ \\
$\ell_1$ & $n$ & $n+m$ & $n+p$ & $n+p+3m$ \\
Elastic net & $n$ & $n+m$ & $n+p$ & $n+p+3m$ \\
Huber & $n$ & $n+2m$ & $n+p$ & $n+p+4m$ \\
$\ell_\infty$ & $n$ & $n+1$ & $n+p$ & $n+p+m+2$ \\
$\ell_2$ & $n$ & $n+m+1$ & $n+p$ & \cellcolor{gray!15}second-order cone \\
\bottomrule
\end{tabular}
\end{table}

By setting $\lambda^0=0$, it follows that $\lambda^k=0$ and
$x^k=\tilde x^k$ for all $k\geq1$. Therefore, \eqref{eq:QP_ADMM} can be simplified to
\begin{subequations}\label{eq:QP_ADMM_solution}
\begin{equation}\label{eq:QP_ADMM_KKT}
\begin{bmatrix}
Q+\sigma I & A^\top\\
A & -\rho^{-1}I
\end{bmatrix}
\begin{bmatrix}
x^{k+1}\\ \nu^{k+1}
\end{bmatrix}
=
\begin{bmatrix}
\sigma x^k-q\\ z^k-\rho^{-1}\mu^k
\end{bmatrix},
\end{equation}
\begin{align}
\tilde z^{k+1}
&=Ax^{k+1}
=z^k+\rho^{-1}(\nu^{k+1}-\mu^k),\\
z^{k+1}
&=\Proj_{[\ell,u]}
\left(\tilde z^{k+1}+\rho^{-1}\mu^k\right), \label{eq:QP_ADMM_solution_proj}\\
\mu^{k+1}
&=\mu^k+\rho\left(\tilde z^{k+1}-z^{k+1}\right).
\end{align}
\end{subequations}

\begin{remark}[Standard Lifting]\label{rem:lifting}
The usual approach lifts \ref{eq:QP_slack} to a standard \ref{eq:QP} by introducing additional decision variables and constraint rows \cite{boyd2004_convex}. Each additional decision variable and constraint increases the KKT matrix in \eqref{eq:QP_ADMM_KKT}. The resulting problem sizes are summarized in Table~\ref{tab:lifting}, and the liftings of common penalties are detailed in Appendix~\ref{app:penalties}.
\end{remark}

\section{ADMM for \ref{eq:QP_slack}} \label{sec:slack_ADMM}

In this section, we show that the ADMM scheme outlined in Section~\ref{sec:QP_ADMM} can be adapted to solve \ref{eq:QP_slack} without introducing additional decision variables, by only modifying the projection step \eqref{eq:QP_ADMM_solution_proj}.

Introducing the auxiliary variable $z\in\R^p$, \ref{eq:QP_slack} is equivalent to
\begin{equation}\label{eq:QP_slack_with_z}
\begin{aligned}
\min_{x,z,\xi}\quad
&\tfrac{1}{2}x^\top Qx+q^\top x
+\Phi(\xi)+I_{[\ell,u]}(z+\xi)\\
\text{s.t.}\quad
&Ax=z.
\end{aligned}
\end{equation}
The slack variable $\xi$ only appears in the last two terms of the objective. Hence, we first eliminate $\xi$, then compute the resulting $z$-update. To this end, we assume throughout that $\Phi$ is unconditional.
\begin{definition}[Unconditional function]\label{def:unconditional}
A function $\Phi:\R^p\to\R\cup\{+\infty\}$ is unconditional if it is
invariant under arbitrary coordinatewise sign changes, i.e.,
\begin{equation}\label{eq:unconditional}
\Phi(D\xi)=\Phi(\xi)
\quad\!\!\text{for all}\quad\!\!
D=\operatorname{diag}(d),\; d\in\{-1,1\}^p.
\end{equation}
\end{definition}
\noindent This class covers a wide range of common penalty functions, including $\ell_k$-norms and separable penalties $\Phi(\xi)=\sum_i\phi_i(\xi_i)$ with even convex functions $\phi_i$ (see Appendix~\ref{app:penalties}).

\begin{lemma}[Slack elimination]\label{lem:slack_elimination}
Let $\Phi$ be proper, closed, convex, and unconditional, with
$\Phi(0)=0$.
Given $z\in\R^p$, define the box projection residual
\begin{equation}\label{eq:box_residual}
R(z):=z-\Proj_{[\ell,u]}(z),
\end{equation}
and the relaxed indicator
\begin{equation}\label{eq:slack_reduced_penalty}
g(z)
:=\inf_{\xi\in\R^p}
\bigl\{\Phi(\xi)+I_{[\ell,u]}(z+\xi)\bigr\}.
\end{equation}
Then
\begin{equation}\label{eq:slack_reduced_penalty_value}
g(z)=\Phi\!\left(-R(z)\right)=\Phi\!\left(R(z)\right).
\end{equation}
Whenever $g(z)<+\infty$, the infimum is attained at
\begin{equation}\label{eq:optimal_slack}
\xi^\star(z):=-R(z).
\end{equation}
\end{lemma}
\begin{proof}
First, we observe that, by convexity and unconditionality,
\begin{equation}\label{eq:absolute_monotonicity}
|a|\leq |b| \quad\Longrightarrow\quad \Phi(a)\leq\Phi(b),
\end{equation}
where the inequality is componentwise, generalizing the classical equivalence of absolute and monotonic norms \cite[Thm.~2]{bauer1961_absolute}. Indeed, $|a|\leq|b|$ implies that $a$ is in the box $[-|b_1|,|b_1|]\times\dots\times[-|b_p|,|b_p|]$, which is the convex hull of its vertices, so $a=\sum_{i=1}^p\lambda_iD_ib$ for some $\lambda_i\geq 0$ satisfying $\sum_{i=1}^p\lambda_i=1$ and $D=\textrm{diag}(d_i)$ for some $d_i\in\{-1,1\}$.  Then, by convexity,
$$
\Phi(a)=
\Phi\bigg(\sum_{i=1}^p\lambda_iD_ib\bigg)
\leq
\sum_{i=1}^p\lambda_i\Phi(Db)
=
\Phi(Db),
$$
so the result \eqref{eq:absolute_monotonicity} follows from unconditionality.
Thus, \begin{equation}\label{eq:Phi_proj<=Phi_y}
\Phi\!\left(\Proj_{[\ell,u]}(z)-z\right)
\mathop{\leq}^{\eqref{eq:absolute_monotonicity}} \Phi(y-z)
\ \text{for any }y\in[\ell,u],
\end{equation}
since $|\Proj_{[\ell,u]}(z)-z|\leq |y-z|$ for any $y\in[\ell,u]$.

Moreover, applying the change of variable $y=z+\xi$,  we can rewrite $g$ as
\begin{align}\label{eq:slack_reduced_penalty:reformulated}
g(z)&\mathop{=}^{\eqref{eq:slack_reduced_penalty}}
\inf_{y\in\R^p}
\Phi(y-z)+I_{[\ell,u]}(y)
=\inf_{y\in[\ell,u]}\Phi(y-z).
\end{align}
Combining  \eqref{eq:Phi_proj<=Phi_y} and \eqref{eq:slack_reduced_penalty:reformulated}, we obtain the first equality in  \eqref{eq:slack_reduced_penalty_value}:
\begin{gather*}
g(z) \mathop{=}^{\eqref{eq:slack_reduced_penalty:reformulated}}\inf_{y\in[\ell,u]}\Phi(y-z) 
\leq
\Phi\big(\Proj_{[\ell,u]}(z)-z\big),
\\
\Phi\big(\Proj_{[\ell,u]}(z)-z\big)
\mathop{\leq}^{\eqref{eq:Phi_proj<=Phi_y}}
\inf_{y\in[\ell,u]}\Phi(y-z)
\mathop{=}^{\eqref{eq:slack_reduced_penalty:reformulated}}
g(z)
\end{gather*}
so $g(z)=\Phi(\Proj_{[\ell,u]}(z)-z) =\Phi(-R(z))$. The second follows similarly from unconditionality. If this value is finite, then
$y^\star=\Proj_{[\ell,u]}(z)$ is a minimizer and gives
\eqref{eq:optimal_slack}.
\end{proof}

Using Lemma~\ref{lem:slack_elimination}, we can eliminate $\xi$ from \eqref{eq:QP_slack_with_z}, so \ref{eq:QP_slack} is equivalent to
\begin{equation}\label{eq:QP_slack_reduced}
\begin{aligned}
\min_{x,z}\quad
&\tfrac{1}{2}x^\top Qx+q^\top x+g(z)\\
\text{s.t.}\quad
&Ax=z,
\end{aligned}
\end{equation}
and by introducing the same splitting as in Section~\ref{sec:QP_ADMM}, we obtain the problem
\begin{equation}\label{eq:QP_slack_OSQP_split}
\begin{aligned}
\min_{\tilde x,\tilde z,x,z}\quad
&\tfrac{1}{2}\tilde x^\top Q\tilde x
+q^\top\tilde x
+I_{A\tilde x=\tilde z}(\tilde x,\tilde z)
+g(z)\\
\text{s.t.}\quad
&(\tilde x,\tilde z)=(x,z).
\end{aligned}
\end{equation}
Comparing \eqref{eq:QP_slack_OSQP_split} with \eqref{eq:QP_OSQP_split}, the two problems are identical except that the indicator function $I_{[\ell,u]}(\cdot)$ has been replaced by the function $g(\cdot)$ from Lemma~\ref{lem:slack_elimination}. Hence, we can follow the same derivation as in Section~\ref{sec:QP_ADMM}, with the only difference that the $z$-update \eqref{eq:z_update} is replaced by
\begin{equation}\label{eq:QP_slack_z_subproblem}
z^{k+1}
=\mathop{\arg\min}_z\left\{
g(z)+\tfrac{\rho}{2}\|z-v^{k+1}\|_2^2\right\}.
\end{equation}
The following theorem reduces the $z$-update to the proximal operator of $\Phi$, which has a closed form for a wide range of penalty functions (see Appendix~\ref{app:penalties}).

\begin{table}[t]
\caption{ADMM iterations for \ref{eq:QP} and \ref{eq:QP_slack}.}
\label{tab:admm_comparison}
\centering
\renewcommand{\arraystretch}{1.3}
\begin{tabular}{lcc}
\toprule
& ADMM for \ref{eq:QP} & ADMM for \ref{eq:QP_slack} \\
\midrule
$x$-update & solve \eqref{eq:QP_ADMM_KKT} & solve \eqref{eq:QP_ADMM_KKT} \\
\rowcolor{gray!15}
$z$-update & $\Proj_{[\ell,u]}(Ax+\mu/\rho)$ & $\Proj^{\Phi}_{[\ell,u]}(Ax+\mu/\rho)$ \\
dual update & $\mu+\rho(Ax-z)$ & $\mu+\rho(Ax-z)$ \\
\bottomrule
\end{tabular}
\end{table}

\begin{theorem}[Generalized slack projection]\label{thm:slack_prox}
Under the assumptions of Lemma~\ref{lem:slack_elimination}, let $g$ be the relaxed indicator function
\begin{equation*}
g(z)
:=\inf_{\xi\in\R^p}
\bigl\{\Phi(\xi)+I_{[\ell,u]}(z+\xi)\bigr\}.
\end{equation*}
For $\rho>0$ and
$v\in\R^p$, let
\begin{equation}
\bar v=\Proj_{[\ell,u]}(v),
\qquad r=R(v)=v-\bar v.
\end{equation}
Then
\begin{equation}\label{eq:slack_prox_solution}
\mathop{\arg\min}_{z\in \R^p}\left\{g(z)+\tfrac{\rho}{2}\|z-v\|_2^2\right\}
=\bar v+\prox_{\Phi/\rho}(r),
\end{equation}
where the proximal operator of $\Phi$ is given by
\begin{equation}\label{eq:prox}
\prox_{\Phi/\rho}(r)
:=\mathop{\arg\min}_s\left\{\Phi(s)+\tfrac{\rho}{2}\|s-r\|_2^2\right\}.
\end{equation}
\end{theorem}

\begin{proof}
Let
\begin{equation}
s^\star:=\prox_{\Phi/\rho}(r)
\end{equation}
and define the box between $0$ and $r$ by
\begin{equation}
B_r:=\prod_{i=1}^p
[\min\{0,r_i\},\max\{0,r_i\}].
\end{equation}
We first observe that \(s^\star\in B_r\). Indeed, for any $s\in\R^p$, let $\bar s=\Proj_{B_r}(s)$.  Coordinatewise
projection gives
\begin{equation}
|\bar s|\leq |s|,
\qquad
\|\bar s-r\|_2\leq\|s-r\|_2,
\end{equation}
with strict inequality in the second relation when $s\notin B_r$.
The objective in \eqref{eq:prox} is finite at $s=0$, so it is finite at
$s^\star$.  If $s^\star\notin B_r$, then
letting $\bar s^\star=\Proj_{B_r}(s^\star)$,
\eqref{eq:absolute_monotonicity} gives
$\Phi(\bar s^\star)\leq\Phi(s^\star)<+\infty$, while the quadratic term
decreases strictly.  This contradicts the optimality of $s^\star$.
Hence
\begin{equation}\label{eq:prox_between_zero_r}
s^\star\in B_r.
\end{equation}

Set $z^\star=\bar v+s^\star$.  For each $i$,
\eqref{eq:prox_between_zero_r} implies: if
$v_i\in[\ell_i,u_i]$, then $\bar v_i=v_i$, $r_i=0$, and $s_i^\star=0$;
if $v_i>u_i$, then $\bar v_i=u_i$, $r_i>0$, and
$s_i^\star\in[0,r_i]$, so $z_i^\star=u_i+s_i^\star\geq u_i$; and if
$v_i<\ell_i$, then $\bar v_i=\ell_i$, $r_i<0$, and
$s_i^\star\in[r_i,0]$, so $z_i^\star=\ell_i+s_i^\star\leq\ell_i$. Therefore,
\begin{equation}\label{eq:candidate_residual}
\Proj_{[\ell,u]}(z^\star)=\bar v,\quad
R(z^\star)=s^\star,\quad
z^\star-v=s^\star-r.
\end{equation}
Since $R(\cdot)$ is componentwise $1$-Lipschitz, we have
\begin{equation}\label{eq:residual_nonexpansive}
\|R(z)-r\|_2 = \|R(z)-R(v)\|_2 \leq \|z-v\|_2,
\end{equation}
and consequently, for every $z$,
\begin{align}
g(z)+\tfrac{\rho}{2}\|z-v\|_2^2
&\geq
\Phi(R(z))+\tfrac{\rho}{2}\|R(z)-r\|_2^2\\
&\geq
\Phi(s^\star)+\tfrac{\rho}{2}\|s^\star-r\|_2^2.
\end{align}
The first inequality follows from Lemma~\ref{lem:slack_elimination} and \eqref{eq:residual_nonexpansive}, the second inequality follows from the definition of $s^\star$, and \eqref{eq:candidate_residual} shows that equality holds at $z=z^\star$.
The function
$(z,\xi)\mapsto\Phi(\xi)+I_{[\ell,u]}(z+\xi)$ is jointly convex, so
partial minimization over $\xi$ shows that $g$ is convex.
Since $g$ is convex, the proximal objective is strongly convex.  Thus its
unique minimizer is $z^\star=\bar v+\prox_{\Phi/\rho}(r)$.
\end{proof}
Motivated by Theorem~\ref{thm:slack_prox}, we define the \emph{generalized slack projection}
\begin{equation}\label{eq:slack_projection}
\Proj^{\Phi}_{[\ell,u]}(v)
:=\Proj_{[\ell,u]}(v)+\prox_{\Phi/\rho}\bigl(R(v)\bigr),
\end{equation}
where $R(\cdot)$ is the box projection residual \eqref{eq:box_residual}.
Hence, the ADMM iterations for \ref{eq:QP_slack} are identical to \eqref{eq:QP_ADMM_solution}, except that the projection step \eqref{eq:QP_ADMM_solution_proj} is replaced by
\begin{equation}\label{eq:QP_slack_z_update}
z^{k+1}=\Proj^{\Phi}_{[\ell,u]}(v^{k+1}),
\end{equation}
with $v^{k+1}=\tilde z^{k+1}+\rho^{-1}\mu^k$. Table~\ref{tab:admm_comparison} summarizes both schemes side by side. In particular, the KKT matrix is unchanged, so its factorization can be reused, and the additional cost per iteration is a single evaluation of $\prox_{\Phi/\rho}$.
Moreover, $g=\Phi\circ R$ is proper ($g=0$ on $[\ell,u]$), closed ($\Phi$ closed, $R$ continuous), and convex (proof of Theorem~\ref{thm:slack_prox}). Hence, \eqref{eq:QP_slack_OSQP_split} has the same two-block structure as \eqref{eq:QP_OSQP_split}, and the standard convergence guarantees of ADMM \cite{boyd2011_admm, stellato2020_osqp} carry over unchanged to the modified iteration.
By Lemma~\ref{lem:slack_elimination}, the eliminated slack associated with $z^{k+1}$ is
\begin{equation}\label{eq:slack_recovery}
\xi^{k+1}=-R(z^{k+1})=-\prox_{\Phi/\rho}\bigl(R(v^{k+1})\bigr).
\end{equation}

\begin{remark}[Relation to known projections]\label{rem:known_projections}
The generalized slack projection contains three known special cases. For the hard-constrained case $\Phi=I_{\{0\}}$, we have $\prox_{\Phi/\rho}=0$, and \eqref{eq:slack_projection} reduces to the standard projection $\Proj_{[\ell,u]}$. For the quadratic penalty $\Phi(\xi)=\tfrac{\alpha}{2}\|\xi\|_2^2$, we have $\prox_{\Phi/\rho}(r)=\tfrac{\rho}{\rho+\alpha}r$, and \eqref{eq:slack_projection} recovers the smoothed projection of \cite{lew2026_slackqp}. For $\Phi(\xi)=\alpha\|\xi\|_2$, \eqref{eq:slack_projection} recovers the known proximal operator of the distance function $\alpha\operatorname{dist}_{[\ell,u]}$ \cite{parikh2014_prox, beck2017_first_order}. Closed-form proximal operators of the other penalties in Table~\ref{tab:lifting} are given in Appendix~\ref{app:penalties}.
\end{remark}

\begin{figure*}[t]
\centering
\includegraphics[width=\textwidth]{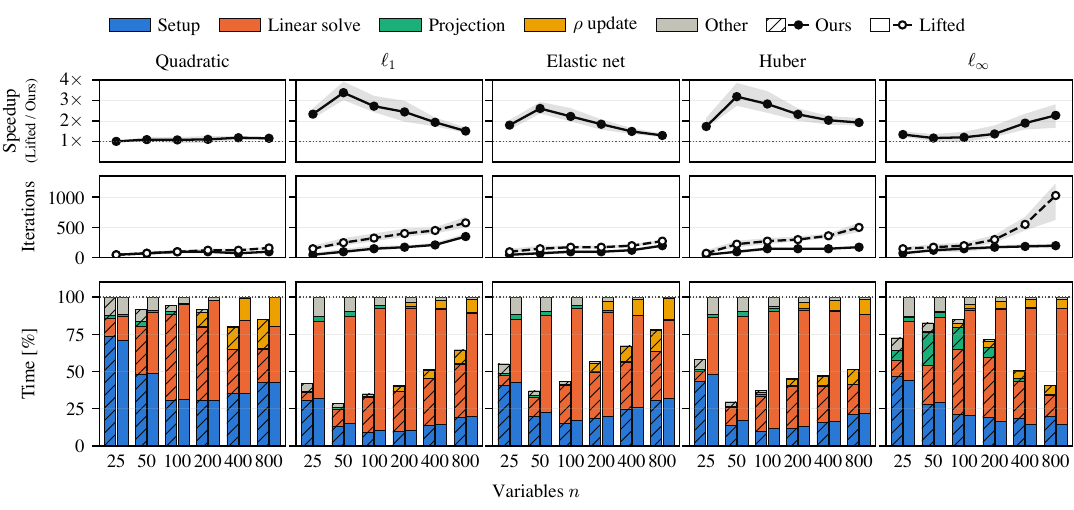}
\caption{Our approach against explicit lifting on random sparse QPs with $p=n$ constraints, all of them soft, for each penalty. Top: speedup, i.e., lifted over our setup-plus-solve time. Middle: ADMM iterations. Bottom: runtime breakdown, normalized to the lifted runtime (hatched: ours). Lines and bars are medians over 100 instances; shaded bands are interquartile ranges.}
\label{fig:scaling}
\end{figure*}
\section{Numerical Results}

We evaluate our approach in two parts. First, we benchmark our approach against lifting \ref{eq:QP_slack} to a standard \ref{eq:QP} as described in Table~\ref{tab:lifting} on randomly generated sparse QPs with different penalty functions (Section~\ref{sec:random_qps}). Second, we apply our approach to a building energy management problem where we mix different penalty functions to model different soft-constraint behaviors (Section~\ref{sec:building}).

Building on the popular open-source solver \texttt{OSQP} \cite{stellato2020_osqp}, we add support for the penalty functions outlined in Appendix~\ref{app:penalties}. All experiments have been implemented in \texttt{Python}, which interfaces with the native \texttt{C} code of the modified \texttt{OSQP} implementation. Everything was evaluated on a MacBook Pro with an M5 Pro processor and 48 GB of RAM. All detailed parameters used in the examples can be found in the code repository.

\subsection{Random Sparse QPs} \label{sec:random_qps}

We generate random sparse QP instances with $n$ decision variables and $p=n$ constraints. The quadratic cost is constructed as $Q=L^\top L + 0.1I_n$, where the regularization $0.1I_n$ ensures positive definiteness, and $L$ is a sparse matrix with Gaussian-sampled entries, such that $Q$ has a density of $8\%$. The constraint matrix $A$ is a sparse matrix with Gaussian-sampled entries of density $5\%$, where each row has at least one entry and all rows are normalized to unit norm. For each problem, we randomly pick an anchor point $x_a \in \R^n$ and target point $x_t \in \R^n$. The linear cost is then constructed as $q=-Qx_t$, ensuring that $x_t$ is the unconstrained minimizer of the nominal cost. For each constraint row $a_i$ of $A$, let $\beta_i=a_i^\top (x_t-x_a)$ be the band distance. We then construct the bounds as $\ell_i=a_i^\top x_a-0.5|\beta_i|$ and $u_i=a_i^\top x_a+0.5|\beta_i|$, centering the constraint around the anchor point, ensuring that the QP is always feasible. Since the target point is outside the constraint band, we also ensure that the optimal solution always has active constraints.

For each problem size $n$ and penalty function listed in Table~\ref{tab:scaling_times}, we generate $100$ random instances. We then take each problem instance and lift it by introducing additional decision variables to form a standard \ref{eq:QP} as summarized in Table~\ref{tab:lifting} and Appendix~\ref{app:penalties}. Note that we are not considering $\ell_2$-norm penalties as they would require a second-order cone solver. While our implementation supports $\ell_2$-norm penalties, the original \texttt{OSQP} solver does not support second-order cone constraints. We solve both the original problem with our approach as well as the lifted problem using the same solver settings with an absolute and relative tolerance of $\epsilon_\text{abs}=\epsilon_\text{rel}=10^{-6}$ and adaptive $\rho$ with a starting value $\rho_0=0.1$.

The results are shown in Fig.~\ref{fig:scaling} and Table~\ref{tab:scaling_times}. The proposed ADMM approach outperforms all lifted formulations by a factor of up to $3.4\times$. The speedup comes mainly from two factors: 1) the size of the KKT system is not increased, leading to faster linear system solves, and 2) the number of iterations is lower for all penalty functions except the quadratic one, while the evaluation time of $\prox_{\Phi/\rho}$ is negligible. The two formulations induce different ADMM splittings and conditioning, resulting in substantially different iteration counts for several penalties. Characterizing this difference theoretically is left for future work.

For the $\ell_1$, elastic net, and Huber penalties, the speedup peaks at medium-sized problems. For small problems, the setup time, which is almost identical for both formulations, dominates and dilutes the per-iteration savings. For large problems, the nonzeros of $Q$ and $A$ grow quadratically in $n$, while lifting only adds sparse rows and columns, so the factorization cost is dominated by the original problem rather than the lifting.
For the $\ell_\infty$ penalty, lifting adds only a single variable, so the per-iteration cost is similar for both formulations, but the lifted iteration count grows rapidly with $n$, increasing the speedup.

\begin{table}[t]
\caption{Median setup-plus-solve time of our approach and of explicit lifting on the random QPs of Fig.~\ref{fig:scaling}, over 100 instances per entry.}
\label{tab:scaling_times}
\centering
\setlength{\tabcolsep}{4pt}
\begin{tabular}{lr@{\,/\,}lr@{\,/\,}lr@{\,/\,}lr@{\,/\,}lr@{\,/\,}l}
\toprule
& \multicolumn{10}{c}{Ours / Lifted Runtime [ms]} \\
\cmidrule(lr){2-11}
Penalty & \multicolumn{2}{c}{$n=25$} & \multicolumn{2}{c}{$n=50$} & \multicolumn{2}{c}{$n=100$} & \multicolumn{2}{c}{$n=400$} & \multicolumn{2}{c}{$n=800$} \\
\midrule
Quadratic & 0.080 & 0.081 & 0.15 & 0.17 & 0.71 & 0.73 & 19 & 23 & 132 & 154 \\
$\ell_1$ & 0.082 & 0.19 & 0.18 & 0.63 & 0.96 & 2.7 & 31 & 61 & 226 & 346 \\
Elastic net & 0.081 & 0.14 & 0.16 & 0.43 & 0.68 & 1.6 & 23 & 34 & 169 & 217 \\
Huber & 0.079 & 0.14 & 0.18 & 0.61 & 0.89 & 2.5 & 25 & 53 & 161 & 310 \\
$\ell_\infty$ & 0.099 & 0.14 & 0.28 & 0.34 & 1.1 & 1.4 & 35 & 65 & 212 & 500 \\
\bottomrule
\end{tabular}

\end{table}

\subsection{Building Energy Management with Mixed Penalties} \label{sec:building}

\begin{figure}[t]
\centering
\includegraphics[width=\columnwidth]{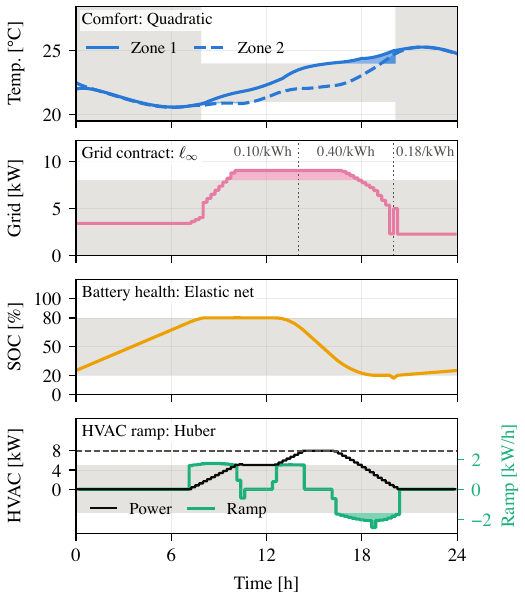}
\caption{One-day building energy management schedule with a different penalty for each soft constraint group. Each panel shows a soft-constrained quantity, its band (gray), and its slack (tinted): zone temperatures against the comfort band (quadratic), grid import against the contracted limit ($\ell_\infty$ over the day; dotted lines separate the tariff periods, labeled with the price per kWh), battery state of charge against its health band (elastic net), and the HVAC ramp rate against its limit (Huber, right axis), together with the HVAC power and its capacity (dashed).}
\label{fig:building}
\vspace{-5mm}
\end{figure}

Realistic optimization problems interacting with the real world involve multiple soft constraints encoding different intents. In this section, we consider a case study of a building energy management optimization problem. We consider a building with two thermal zones which are coupled through a door. Zone 1 is exposed to the outside on a warm summer day, while Zone 2 has an HVAC system which can provide cooling. Additionally, the building has a base load $L_k$ of 4\,kW during occupancy (8--20\,h) and 2\,kW otherwise, a battery with a capacity of 20\,kWh, and a grid connection with dynamic energy prices and a grid contract limit of 8\,kW.

We solve an MPC problem over a single day with $N=96$ steps and a time discretization of $\Delta t=15$\,min, minimizing the electricity cost:
\begin{subequations}\label{eq:building}
\begin{align}
\min_{T,b,h,d,g,\xi}\;
&\sum_{k=0}^{N-1}\bigl(\Delta t\,\pi_k g_k+\varepsilon\|(h_k,d_k,g_k)\|_2^2\bigr)+\sum_{j}\Phi_j(\xi^j) \label{eq:building_cost}\\
\text{s.t.}\quad
&T_{k+1}=A_T T_k+B_T h_k+w_k, \label{eq:building_thermal}\\
&b_{k+1}=b_k-\Delta t\,d_k, \label{eq:building_battery}\\
&g_k+d_k=h_k+L_k, \label{eq:building_balance}\\
&(b_k,h_k,d_k,g_k)\in\mathcal{X}, \label{eq:building_hard}\\
&\underline{T}_k\leq T_k+\xi^c_k\leq\overline{T}_k, \label{eq:building_comfort}\\
&g_k+\xi^g_k\leq\bar g, \label{eq:building_grid}\\
&\underline{b}\leq b_k+\xi^b_k\leq\overline{b}, \label{eq:building_soc}\\
&-\bar r\leq(h_k-h_{k-1})/\Delta t+\xi^r_k\leq\bar r, \label{eq:building_ramp}
\end{align}
\end{subequations}
where $T_k\in\R^2$ are the zone temperatures, $b_k$ is the battery charge, $h_k$ is the HVAC power, $d_k$ is the battery discharge power, $g_k$ is the grid import, $\pi_k$ is the time-of-use price, and $\varepsilon=10^{-3}$ is a small regularization weight. The thermal dynamics \eqref{eq:building_thermal} follow from a zero-order-hold discretization of a two-zone thermal model, where $w_k$ collects the outside temperature and internal heat gains. \eqref{eq:building_balance} ensures that electrical power is balanced, and the hard constraints \eqref{eq:building_hard} encode physical limits such as maximum HVAC power, and battery power and capacity. The soft constraints \eqref{eq:building_comfort}--\eqref{eq:building_ramp} are penalized by $\Phi_j$, $j\in\{c,g,b,r\}$, as follows:

\noindent \textbf{Comfort (quadratic)}: $\Phi_c(\xi^c)=0.25\|\xi^c\|_2^2$\\
Zone temperatures should remain within $[21,24]\,^\circ$C during occupancy (8--20\,h), and $[11,34]\,^\circ$C otherwise. Small violations are cheap while large violations are expensive.

\noindent \textbf{Grid contract ($\boldsymbol{\ell_\infty}$)}: $\Phi_g(\xi^g)=0.9\|\xi^g\|_\infty$\\
Violating the grid contract of $\bar g=8$\,kW incurs a peak-demand charge, i.e., only the maximum violation over the day is penalized.

\noindent \textbf{Battery health (elastic net)}: $\Phi_b(\xi^b)=0.05\|\xi^b\|_1+\tfrac{1}{12}\|\xi^b\|_2^2$\\
To reduce battery degradation, the state of charge should be in $[20,80]\,\%$. This should hold exactly unless violating it is necessary, and then only slightly.

\noindent \textbf{HVAC ramp (Huber)}: $\Phi_r(\xi^r)=\sum_k\phi(\xi^r_k)$\\
To reduce wear, there is an HVAC power rate limit of $\bar\omega=1.6$\,kW/h. We tolerate small violations while limiting the cost of large ones with the Huber penalty $\phi$ (Appendix~\ref{app:huber}) with $\alpha=0.125$ and $\delta=0.5$.

The resulting schedule can be seen in Fig.~\ref{fig:building}. One can see the trade-offs of the different penalties. The comfort violation is distributed over many timesteps. A peak-demand charge is only paid once; hence, the grid import stays high over multiple timesteps. Finally, the elastic net penalty imposes sparse constraint violations, conserving the battery as much as possible.
The problem is solved to an accuracy of $10^{-6}$ in 5.1\,ms. For comparison, we solve it with the standard approach of lifting the problem to a \ref{eq:QP}, which takes 14.5\,ms, i.e., $2.83\times$ longer. This comparison shows the benefits of having a solver which natively supports a wide range of penalty functions for soft constraints.

\printbibliography

\appendices

\section{Common Slack Penalties} \label{app:penalties}

For common slack penalties, we list the lifting to a standard \ref{eq:QP} \cite{boyd2004_convex}, which yields the problem sizes in Table~\ref{tab:lifting}, and the proximal operator \cite[Chapter~6]{parikh2014_prox}, which defines the generalized slack projection \eqref{eq:slack_projection}. All penalties are unconditional, so Lemma~\ref{lem:slack_elimination} and Theorem~\ref{thm:slack_prox} apply. For separable penalties $\Phi(\xi)=\sum_i\phi(\xi_i)$ with even $\phi$, the proximal operator applies componentwise.

For the liftings, let $a_i^\top$ denote the $i$-th row of $A$ and $m$ the number of soft constraints. Each additional decision variable and constraint increases the order of the KKT matrix in \eqref{eq:QP_ADMM_KKT} by one. Further, let
\begin{equation}\label{eq:soft_thresholding}
\mathcal S_\tau(r)
:=\operatorname{sign}(r)(|r|-\tau)_+
\end{equation}
be the soft thresholding operator.

\subsection{Quadratic Penalty}
For $\alpha>0$,
\begin{subequations}
\begin{align}
\phi(s)&=\tfrac{\alpha}{2}s^2,\\
\prox_{\phi/\rho}(r)&=\frac{\rho}{\rho+\alpha}r.
\end{align}
\end{subequations}
\textit{Lifting}: A decision variable $\xi_i$ with cost $\tfrac{\alpha}{2}\xi_i^2$ relaxes the soft constraint to $\ell_i\leq a_i^\top x+\xi_i\leq u_i$, adding one decision variable and no constraints.

\subsection{\texorpdfstring{$\ell_1$}{l1} Penalty}
For $\alpha>0$,
\begin{subequations}
\begin{align}
\phi(s)&=\alpha|s|,\\
\prox_{\phi/\rho}(r)&=\mathcal S_{\alpha/\rho}(r).
\end{align}
\end{subequations}
\textit{Lifting}: A decision variable $s_i\geq0$ with cost $\alpha s_i$ relaxes the soft constraint to $\ell_i-s_i\leq a_i^\top x\leq u_i+s_i$. Since $s_i$ enters both bounds with opposite signs, the constraint is split into two one-sided constraints, and $s_i\geq0$ adds a third constraint, i.e., one additional variable and two additional constraints.

\subsection{Elastic Net Penalty}
For $\alpha_1,\alpha_2>0$,
\begin{subequations}
\begin{align}
\phi(s)&=\alpha_1|s|+\tfrac{\alpha_2}{2}s^2,\\
\prox_{\phi/\rho}(r)
&=\frac{\rho}{\rho+\alpha_2}
\mathcal S_{\alpha_1/\rho}(r).
\end{align}
\end{subequations}
\textit{Lifting}: As for the $\ell_1$ penalty, but with cost $\alpha_1 s_i+\tfrac{\alpha_2}{2}s_i^2$.

\subsection{Huber Penalty} \label{app:huber}
For $\alpha,\delta>0$,
\begin{subequations}
\begin{equation}
\phi(s)=
\begin{cases}
\tfrac{\alpha}{2}s^2, & |s|\leq\delta,\\
\alpha\bigl(\delta|s|-\tfrac{1}{2}\delta^2\bigr), & |s|>\delta,
\end{cases}
\end{equation}
and
\begin{equation}
\prox_{\phi/\rho}(r)=
\begin{cases}
\dfrac{\rho}{\rho+\alpha}r,
& |r|\leq(1+\alpha/\rho)\delta,\\[2mm]
r-\dfrac{\alpha\delta}{\rho}\operatorname{sign}(r),
& |r|>(1+\alpha/\rho)\delta.
\end{cases}
\end{equation}
\end{subequations}
\textit{Lifting}: The Huber penalty can be written as
\begin{equation}
\phi(\xi_i)=\min_{w_i,\,t_i\geq0}
\bigl\{\tfrac{\alpha}{2}w_i^2+\alpha\delta t_i
: |\xi_i|\leq w_i+t_i\bigr\}.
\end{equation}
Hence, two decision variables $w_i$ and $t_i\geq0$ relax the soft constraint to $\ell_i-w_i-t_i\leq a_i^\top x\leq u_i+w_i+t_i$, which is split into two one-sided constraints as for the $\ell_1$ penalty. Together with $t_i\geq0$, this adds two variables and two constraints.

\subsection{Infinity-Norm Penalty}
For $\alpha>0$, let
\begin{subequations}
\begin{equation}
\Phi(\xi)=\alpha\|\xi\|_\infty.
\end{equation}
By Moreau decomposition,
\begin{equation}\label{eq:inf_norm_moreau}
\prox_{\Phi/\rho}(r)
=r-\Proj_{\{\|\cdot\|_1\leq\alpha/\rho\}}(r),
\end{equation}
so $\prox_{\Phi/\rho}(r)=0$ if $\|r\|_1\leq\alpha/\rho$, and otherwise
\begin{equation}
[\prox_{\Phi/\rho}(r)]_i
=\operatorname{sign}(r_i)\min\{|r_i|,\kappa\},
\end{equation}
where $\kappa>0$ is the unique solution of
\begin{equation}\label{eq:inf_norm_threshold}
\sum_{i=1}^p\bigl(|r_i|-\kappa\bigr)_+=\frac{\alpha}{\rho}.
\end{equation}
\end{subequations}
The left-hand side of \eqref{eq:inf_norm_threshold} is piecewise linear
and strictly decreasing on $[0,\|r\|_\infty]$, so $\kappa$ can be computed
by bisection
\cite[\S6.5.2]{parikh2014_prox}, by sorting $|r|$ in $O(p\log p)$
operations, or in expected $O(p)$ operations by median-based selection
\cite{duchi2008_l1ball}.

\textit{Lifting}: A single decision variable $s\geq0$ with cost $\alpha s$, shared by all soft constraints, relaxes each soft constraint to $\ell_i-s\leq a_i^\top x\leq u_i+s$. Each soft constraint is split into two one-sided constraints, and $s\geq0$ adds one constraint, i.e., one additional variable and $m+1$ additional constraints.

\subsection{Euclidean-Norm Penalty}
For $\alpha>0$, let
\begin{subequations}
\begin{equation}
\Phi(\xi)=\alpha\|\xi\|_2.
\end{equation}
\begin{equation}
\prox_{\Phi/\rho}(r)=
\begin{cases}
\left(1-\dfrac{\alpha}{\rho\|r\|_2}\right)r,
& \|r\|_2>\dfrac{\alpha}{\rho},\\[2mm]
0, & \|r\|_2\leq\dfrac{\alpha}{\rho}.
\end{cases}
\end{equation}
\end{subequations}
\textit{Lifting}: Decision variables $\xi_i$ relax each soft constraint to $\ell_i\leq a_i^\top x+\xi_i\leq u_i$, and an epigraph variable $t$ with cost $\alpha t$ satisfies $\|\xi\|_2\leq t$. This adds $m+1$ variables, but the second-order cone constraint cannot be expressed in a standard \ref{eq:QP}.

\end{document}